\documentclass[12pt,reqno]{article}
\usepackage{amssymb}
\usepackage{amsmath}
\usepackage{amsthm}
\usepackage{hyperref}

\usepackage{enumerate}
\usepackage{color}

\usepackage{graphics}
\usepackage{graphicx}
\usepackage{subfigure}
\usepackage{float}
\usepackage{epstopdf}
\newcounter{assumption}

\allowdisplaybreaks

\newtheorem{remark}{Remark}

\newtheorem{lemma}{Lemma}

\newtheorem{thm}{Theorem}[section]

\newtheorem{prop}{Proposition}[section]

\newtheorem{cor}{Corollary}[section]

\numberwithin{equation}{section} \numberwithin{theorem}{section}
\numberwithin{lemma}{section} \numberwithin{corollary}{section}
\numberwithin{definition}{section}
\numberwithin{proposition}{section} \numberwithin{remark}{section}
\numberwithin{example}{section}

\renewcommand{\oddsidemargin}{5mm}

\def\R{\mathbb R}

\def\tPsi{\widetilde{\Psi}}
\def\tPhi{\widetilde{\Phi}}

\def\omu{\overline{\mu}}

\def\cn{c^*_n}

\makeatletter
\def\@makefnmark{}

\makeatother

\begin{document}

\title{Analysis of Spreading Speeds with an Application to Cellular Neural Networks}

 \author{Zhi-Xian Yu  \\
 \ \\
 {\small \it College of Science, University of Shanghai for Science and Technology}\\
 {\small\it  Shanghai, 200093, China}\\
 {\small \tt zxyu@usst.edu.cn}\\
 \ \\
Lei Zhang\\
 {\small \it School of Mathematical Sciences, University of Science and Technology of China,}\\
{\small \it Hefei, Anhui 230026, China}\\
{\small\it  Department of Mathematics and Statistics, Memorial University of Newfoundland}\\
{\small \it  St. John's, NL AIC5S7, Canada} \\
{\small \tt zhanglei890512@gmail.com}
}

\date{}
\maketitle

 \paragraph{Abstract.}In this paper, we focus on some properties of the spreading speeds which can be estimated by linear operators approach, such as the sign, the continuity and a limiting case which admits no spreading phenomenon.  These theoretical results are well applied to study the effect of templates on propagation speeds for cellular neural networks (CNNs), which admit three kinds of propagating phenomenon. 
\vspace{0.25cm}\\
{\bf{ Keywords.}} Spreading speeds, propagating phenomenon, monotone semiflows, CNNs
\vspace{0.25cm}\\
{\bf{AMS subject classifications.}}  35C07, 34A33, 94C99
\\

\date{}

\section{Introduction}
In the pioneering works of Fisher \cite{fisher1937wave} and Kolmogorov, Petrovskii, and Piskunov
\cite{kolmogorov1937etude}, it was proved that the Fisher's equation with the spatial diffusion $u_t = u_{xx} + u(1-u)$ for $x \in \mathbb{R}$
admits a minimal wave speed $c_{min} = 2$ in the sense that there exits a traveling
wave solution with speed $c$ if and only if $c \geq c_{min}$. Fisher \cite{fisher1937wave} also conjectured that
$c^*= c_{min}$ is the spreading speed of this equation. Aronson and Weinberger \cite{aronson1975nonlinear,aronson1978multidimensional} proved this conjecture for equations with more general monostable nonlinearities. Since then, lots of works have shown the coincidence of the spreading speed with the minimal speed for travelling waves under appropriate assumptions for various evolution systems. Weinberger \cite{weinberger1982long} and Lui \cite{lui1989biological} established the theory of spreading speeds and monostable traveling waves for monotone (order-preserving) operators. This theory has been greatly developed recently in \cite{ding2015principal,fang2014traveling,li2005spreading,liang2006spreading,liang2007asymptotic,liang2010spreading} to monotone semiflows so that it can be applied to various discrete- and continuous-time evolution equations admitting the comparison principle.

When the spreading speeds can be estimated by linear operators approach, Liang and Zhao\cite{liang2007asymptotic} obtained the formula to compute it under the sufficient condition that the infimum of the function $\Phi(\mu)$ is attained at some finite value $\mu^*$ and $\Phi(+\infty)>\Phi(\mu^*)$($\Phi(\mu)$ is from \cite[Section 3]{liang2007asymptotic}). Recently, Ding and Liang \cite{ding2015principal} proved that the formula also holds for the case where $\Phi(+\infty)=\Phi(\mu^*)$. Actually, in most of the earlier works on spreading speeds, various evolution systems satisfy $\Phi(+\infty)=+\infty$, so the infimum can be obtained at some finite value. However, there is no result for the system with $\Phi(+\infty)<+\infty$. It is worth pointing out that the infimum of $\Phi(\mu)$ can be attained at positive infinity in this case. In this paper, we will investigate properties of $\Phi(\mu)$, and apply them to a cellular neural network(Here, $\Phi(+\infty)<+\infty$ can be occurred under some suitable parameters).

Cellular neural networks (CNNs for short), were first introduced in 1988 by Chua and Yang \cite{chua1988cellular,chua1988cellular1} as a novel class of
information processing systems, which possesses some of the key features
of neural networks (NNs) and which has important potential applications in such areas as image processing and pattern recognition  (see, e.g., \cite{chua1998cnn,chua1988cellular,chua1988cellular1}). CNN is simply an analogue dynamic processor
array, made of cells, which contain linear capacitors, linear
resistors, linear and nonlinear controlled sources. This circuit has been used sometimes to test the circuit robustness as well as for implementing the simplest propagating template. The circuit model of a one-dimensional simple CNN without input terms is
\begin{align}\label{1.3}
\frac{dx_i(t)}{dt} =-x_i(t) +\alpha f(x_{i-1}(t))+af(x_{i}(t))+\beta f(x_{i+1}(t)),\,\,i\in\mathbb{Z},
\end{align}
where the output
function $f$ (a nonlinearity) is given by
\begin{align}\label{1.2}
f(u)=\frac{1}{2}(|u+1|-|u-1|).
\end{align}
Here the node voltage $x_i$  at $i$ is called the state of the cell at $i$, $A:=[\alpha,a,\beta]$ constitutes the so-called cloning template, which measures the coupling weights and specifies the interaction between each cell and all its neighbor cells in terms of their state and output variables. That is, $\alpha$ is the interaction from $x_{i-1}$ to $x_i$, so it can be regarded as the rightward interaction. Similarly, $\beta$ is the leftward interaction and $a$ is the own evolution action.

In CNNs, some experimental studies have revealed the propagation
of traveling bursts of activity in slices of excitable neural tissue (see, e.g. \cite{golomb1997propagating,golomb1996propagation,perez1992propagation} ). The underlying mechanism for propagation of these waves (i.e. travelling waves) is thought to be synaptic in origin rather than diffusive as in the propagation of action potentials. Based on the work of \cite{ding2015principal,liang2007asymptotic,liang2010spreading}, the existence of spreading speeds for \eqref{1.3} can be obtained and the spreading phenomenon appears under the assumption $\alpha + \beta + a>1$(Proposition \ref{prop:both>0}). And then, we will analyze the effect of templates on propagation speeds for cellular neural networks (CNNs). It is surprising that the spreading speeds may be less than zero for some special template cases (see Tables \ref{tab:alpha>0&beta>0} and \ref{tab:alpha||beta=0}).
\begin{table}[h]
  \centering
\begin{tabular}{|c|c|c|c|}
\hline
 & $\alpha>\beta>0$ & $\alpha=\beta>0$ & $\beta>\alpha>0$ \\
\hline
$2\alpha^{\frac{1}{2}}\beta^{\frac{1}{2}}+a-1>0$ & $c^*_+>c^*_->0$ & $c^*_+=c^*_->0$ & $c^*_->c^*_+>0$ \\
\hline
$2\alpha^{\frac{1}{2}}\beta^{\frac{1}{2}}+a-1=0$ & $c^*_+>c^*_-=0$ & Impossible & $c^*_->c^*_+=0$ \\
\hline
$2\alpha^{\frac{1}{2}}\beta^{\frac{1}{2}}+a-1<0$ & $c^*_+>0>c^*_-$  & Impossible & $c^*_->0>c^*_+$ \\
\hline
\end{tabular}
  \caption{The cases of $\alpha+\beta+a>1$, $\alpha>0$ and $\beta>0$.\label{tab:alpha>0&beta>0}}
\end{table}
\begin{table}[h]
  \centering
\begin{tabular}{|c|c|c|c|}
\hline
& $\alpha>\beta=0$  & $\beta>\alpha=0$ \\
\hline
$a\geq 1$ & $c^*_+>c^*_-=0$ & $c^*_->c^*_+=0$ \\
\hline
$0 \leq a < 1$ & $c^*_+>0>c^*_-$  &$c^*_->0>c^*_+$ \\
\hline
\end{tabular}
 \caption{The cases of $\alpha+\beta+a>1$, $\alpha=0$ or $\beta=0$.   \label{tab:alpha||beta=0}}
\end{table}

From Table 1, \eqref{1.3} admits three propagation phenomenon: the signals will transfer to both sides, transfer to one side and stay on the other side, transfer to one side and diminish on the other side. We remark that $2\alpha^{\frac{1}{2}}\beta^{\frac{1}{2}}+a-1\leq0$ and $\alpha=\beta$ imply that $\alpha + \beta + a \leq 1 $, which contradicts with $\alpha + \beta + a >1 $, i.e. the spreading speeds do not exist in this case. On the other hand, we can see if the right interaction $\alpha$ is larger than the left one $\beta>0$, then the right spreading speed $c^*_+$ is greater than the left one $c^*_-$. It is worth pointing that $c^*_+>0$ in this case. Moreover,  in addition, if $2 \alpha^{\frac{1}{2}}\beta^{\frac{1}{2}}+a<1$, then it is a surprising phenomenon that the leftward spreading speed $c^*_-<0$, which tells that all signals transfer to the left side and diminish on the other side.

From Table 2, \eqref{1.3} only admits two propagation phenomenon: transfer to one side and stay on the other side, transfer to one side and diminish on the other side. It is not hard to understand that the left or right spreading speed is zero  when the leftward or right interaction is diminished and the own evolution interaction $a\geq 1$. It is interesting to find that leftward or rightward spreading speed may be less than zero when the leftward or rightward interaction is diminished and the own evolution interaction small enough, that is, $0\leq a<1$.

For a more complex problem, the above method may be not very useful. But it is still an interesting problem that how to find a suitable parameters such that $c^*_-<0$. In order to obtain the values of spreading speeds on the neighborhood of a limiting case which admits no spreading phenomenon, we try to prove the continuity of the spreading speeds and investigate a limiting case which admits no spreading phenomenon. That is, for the model \eqref{1.3}, consider the limiting case $\alpha+\beta+a=1$, $\alpha>0$ and $\beta>0$. From the analysis of this limiting case, we can find that there is a suitable parameters such that the $c^*_-$ less than zero.


The remaining part of the paper is organized as follows. Section 2 is devoted to establishing a generalized method to analysis some properties of the spreading speeds, which can be estimated by linear operators approach such as the sign, the continuity and a limiting case which admits no spreading phenomenon. In Section 3, these above theoretic results are applied to the CNNs model.
\section{Basic discussion}
\subsection{Threshold-type conclusion of the sign of $c^*$}
\noindent

Let $\lambda(\mu)$ be a function in $C^2([0,+\infty))$ with the following properties.
\newcounter{Hyp}
\setcounter{Hyp}{0}
\begin{enumerate}[(L1)]
\item $\lambda(\mu)>0$ for any $\mu \in [0,+\infty)$.
\item  $\lambda(0)>1$.
\item  $\ln \lambda(\mu)$ is convex with respect to $\mu\in [0,+\infty)$.
\setcounter{Hyp}{\value{enumi}}
\end{enumerate}

Then we can define $\Phi(\mu)$ and $\Psi(\mu)$ as following:
$$
\Phi(\mu):=\frac{\ln \lambda(\mu)}{\mu},~\mu\in(0,+\infty)\qquad
\Psi(\mu):=\frac{\lambda'(\mu)}{\lambda(\mu)},~\mu\in[0,+\infty)
$$
The following result is from \cite[Lemma 3.8]{liang2007asymptotic}.
\begin{lemma}\label{lem:basic_phi_psi}The following statements hold.
\begin{enumerate}[(i)]
\item\label{lem:item:tPhi:0} $\Phi(\mu)\rightarrow +\infty $ as $\mu\downarrow 0$ .
\item\label{lem:item:tPhi:0:decreasing} $\Phi(\mu)$ is decreasing with  respect to $\mu$ near $0$.
\item\label{lem:item:tPhi_mu} $\Phi'(\mu)$ changes sign at most once on $(0,+\infty)$.
\item\label{lem:item:tPsi} $\Psi(\mu)$ is increasing with respect to $\mu \in[0,+\infty)$ and $\lim\limits_{\mu\rightarrow +\infty} \Phi(\mu)=\lim\limits_{\mu\rightarrow +\infty} \Psi(\mu)$, where the limits may be infinite.
\end{enumerate}

\end{lemma}
Now we define
$$
c^*:=\inf_{\mu>0}\Phi(\mu).
$$
For convenience, we denote
$$
\Phi(+\infty):=\lim\limits_{\mu\rightarrow +\infty} \Phi(\mu)\qquad \text{and} \qquad
\Psi(+\infty):=\lim\limits_{\mu\rightarrow +\infty} \Psi(\mu).
$$
\begin{prop}\label{prop:Phi_Psi}
	There exists $\mu^*\in(0,+\infty]$ such that  $c^*=\Phi(\mu^*)=\Psi(\mu^*)$.
\end{prop}
\begin{proof}
Obviously, there exists $\mu^*\in (0,+\infty]$ such that $c^*=\Phi(\mu^*)$.
 If $\mu^*$ is finite, we deduce that $\Phi'(\mu^*)=0$ as $\Phi(\mu^*)=\inf\limits_{\mu>0}\Phi(\mu)$. By virtue of $\mu\Phi'(\mu^*)=\Psi(\mu^*)-\Phi(\mu^*)$, we conclude that $\Phi(\mu^*)=\Psi(\mu^*)$. If $\mu^*$ is infinite, Lemma \ref{lem:basic_phi_psi}$(iv)$ shows that $\Phi(\mu^*)=\Psi(\mu^*)$. This completes the proof.
\end{proof}
\begin{remark}\label{rem:c^*&c_0}
If $\Phi(+\infty)$ is a finite constant, then $c^*-c_0=\inf\limits_{\mu>0} \frac{\ln(e^{-c_0\mu}\lambda (\mu))}{\mu}$ and $e^{-c_0\mu}\lambda (\mu)$ satisfies {\rm (L1)--(L3)}, where  $c_0=\Phi(+\infty)$.
\end{remark}
Without loss of generality, if $\Phi(+\infty)$ is a finite constant, we can assume that
\begin{enumerate}[(L1)]
\setcounter{enumi}{\value{Hyp}}
\item $\Phi(+\infty)=0.$
\end{enumerate}

If $\lim\limits_{\mu\rightarrow +\infty} \Phi(\mu)$ is infinite, it follows from Lemma \ref{lem:basic_phi_psi} that $\lim\limits_{\mu\rightarrow +\infty} \Phi(\mu)=+\infty$ and we set
\begin{enumerate}[(L1$'$)]
\setcounter{enumi}{\value{Hyp}}
\item $\Phi(+\infty) =+\infty.$
\setcounter{Hyp}{\value{enumi}}
\end{enumerate}

\begin{lemma}
Assume that {\rm (L4)} hold. Then $\lambda'(\mu)\leq 0$ for any $\mu \in [0,+\infty)$ and
 $\lambda(+\infty):=\lim\limits_{\mu \rightarrow +\infty} \lambda(\mu)$ exists.
\end{lemma}
\begin{proof}
Since $\Psi(\mu)$ is increasing with respect to $\mu$, we derive that $\Psi(\mu)\leq \lim\limits_{\mu\rightarrow +\infty} \Psi(\mu)=0$ and $\lambda'(\mu)=\Psi(\mu)\lambda(\mu) \leq 0$ for any $\mu \in (0,+\infty)$. Therefore, $\lim\limits_{\mu \rightarrow +\infty} \lambda(\mu)$ exists .
\end{proof}

\begin{thm}\label{thm1.1}
Assume that {\rm (L4)} hold. Then the following assertions hold:
\begin{enumerate}[(i)]
\item If $\lambda(+\infty)\geq 1$, then $c^*=0$.
\item If $\lambda(+\infty)< 1$, then $c^*<0$.
\end{enumerate}
\end{thm}
\begin{proof}$(i)$. By virtue of $\lambda(+\infty)\geq 1$ and $\lambda'(\mu) \leq 0$ for any $\mu \in [0,+\infty)$, we deduce that  $\Phi(\mu)=\frac{\ln \lambda (\mu)}{\mu}\geq 0$ for any $\mu \in [0,\infty) $. Therefore, $c^* = 0$.

$(ii)$. Since  $\lambda(+\infty) < 1$, it follows that there exists $\mu_0\in [0,+\infty)$ such that $\lambda(\mu_0)<1$. We conclude $c^*\leq \Phi(\mu_0)=\frac{\ln\lambda(\mu_0)}{\mu_0}<0$.
\end{proof}

For any $\mu \in[0,+\infty)$, let
$$
h(\mu)=\ln \lambda (\mu),~ \qquad
g(\mu)=h' (\mu) \mu - h(\mu).~
$$
\begin{thm}\label{thm1.2}
Assume that {\rm (L4$'$)} hold. Then there is $\mu^*\in (0,+\infty)$ such that $g(\mu^*)=0$ and
\begin{enumerate}[(i)]
\item If $h(\mu^*)>0$, then $c^*>0$.
\item If $h(\mu^*)=0$, then $c^*=0$.
\item If $h(\mu^*)<0$, then $c^*<0$.
\end{enumerate}
\end{thm}
\begin{proof}
By virtue of Lemma \ref{lem:basic_phi_psi}$(i) (iii)$ and {\rm (L4$'$)}, we have there is $\mu^*\in (0,+\infty)$ such that $c^*=\inf_{\mu>0} \Phi(\mu)=\Phi(\mu^*)$ and $\Phi'(\mu^*)=0$, which implies that $g(\mu^*)=0$. The rest parts can be easily checked by $c^*=\Phi(\mu^*)=\frac{h(\mu^*)}{\mu^*}$.
\end{proof}
\begin{remark}
If $\lambda(\mu)$ satisfies {\rm (L4$'$)} and $\ln \lambda(\mu)$ is strictly convex in $[0,\infty)$, then $\Psi(\mu)$ is strictly increasing in $[0,\infty)$, which implies that there exists a unique $\mu^*>0$ such $c^*=\Phi(\mu^*)$.
\end{remark}
\begin{cor}\label{cor:sign}Assume that {\rm (L4$'$)} hold and let $h_0:=\inf\limits_{\mu>0}h(\mu)(=\min\limits_{\mu\geq 0} h(\mu))$. The following conclusions hold.
\begin{enumerate}[(i)]
\item If $h_0>0$ then $c^*>0$.
\item If $h_0=0$ then $c^*=0$.
\item If $h_0<0$ then $c^*<0$.
\end{enumerate}
\end{cor}
\begin{proof}
We choose $\mu^*\in (0,+\infty)$ such that $\inf_{\mu>0} \Phi(\mu)=\Phi(\mu^*)$. The infimum of $h(\mu)$ on $(0,+\infty)$ can be obtained at $0\leq \mu_0<+\infty$ due to {\rm (L4$'$)}.
In the case where $h_0>0$, $c^*=\Phi(\mu^*)=\frac{h(\mu^*)}{\mu^*}\geq \frac{h_0}{\mu^*}>0.$
In the case where $h_0\leq 0$, we have $\mu_0>0$ from $h(0)>0$.
If $h_0=0$, we deduce $c^*=0$ from  $0=\frac{h_0}{\mu^*}\leq \frac{h(\mu^*)}{\mu^*}=\Phi(\mu^*)\leq\Phi(\mu_0)=0$.
If $h_0<0$, then $c^*=\Phi(\mu^*)\leq \Phi(\mu_0)<0.$

\end{proof}

\subsection{Continuity of $c^*$ from above and below}
Throughout this subsection, let $\lambda \in C^2([0,+\infty))$ and $\lambda_n\in C^2([0,+\infty))$ satisfy {\rm (L1)--(L3)} and $\lambda_n(\mu)$ converges to $\lambda(\mu)$ as $n \rightarrow +\infty$ for any $\mu\in [0,+\infty)$ from above or below. Then $\Phi(\mu)$, $\Psi(\mu)$, $\Phi_n(\mu)$, $\Psi_n(\mu)$, $c^*$ and $c^*_n$ can be defined as follows:
$$
\Phi(\mu):=\frac{\ln \lambda(\mu)}{\mu},~\mu\in(0,+\infty),\qquad
\Psi(\mu):=\frac{\lambda'(\mu)}{\lambda(\mu)},~\mu\in[0,+\infty),
$$
$$
\Phi_n(\mu):=\frac{\ln \lambda_n(\mu)}{\mu},~\mu\in(0,+\infty),\qquad
\Psi_n(\mu):=\frac{\lambda_{n}'(\mu)}{\lambda_n(\mu)},~\mu\in[0,+\infty),
$$

$$
c^*:=\inf_{\mu>0}\Phi(\mu),\qquad
c^*_n:=\inf_{\mu>0}\Phi_n(\mu).
$$
By Dini's theorem, we have the following consequence.
\begin{lemma}\label{lem:Psi_Phi:con}
Assume that $\lambda_n (\mu) \geq \lambda_{n+1} (\mu)\geq \lambda (\mu)$ or $\lambda_n (\mu) \leq \lambda_{n+1} (\mu) \leq \lambda (\mu)$ for all $n\geq 1$ and $\mu \in [0,+\infty)$. Then $\Phi_n(\mu)$ converges to $\Phi(\mu)$ uniformly as $n \rightarrow +\infty$ on any closed bounded subset of $(0,+\infty)$.
\end{lemma}
	
We first discuss the case where  $\inf\limits_{\mu>0} \Phi(\mu)<\Phi(+\infty)$.
\begin{prop}\label{prop:conti:inf<}
Assume that $\lambda_n (\mu) \geq \lambda_{n+1} (\mu)\geq \lambda (\mu)$ or $\lambda_n (\mu) \leq \lambda_{n+1} (\mu) \leq \lambda (\mu)$ for all $n\geq 1$ and $\mu \in [0,+\infty)$.
If there exists  $\inf\limits_{\mu > 0}\Phi(\mu)<\Phi(+\infty)$, then $\lim\limits_{n\rightarrow +\infty} c_n^*=c^*$.
\end{prop}
\begin{proof}
Since $\inf\limits_{\mu > 0}\Phi(\mu)<\Phi(+\infty)$ and $\lambda(\mu)$ satisfies {\rm (L1)-(L3)}, there exists $\mu^*\in (0,+\infty)$ such that
$c^*=\inf\limits_{\mu>0}\Phi(\mu)=\Phi(\mu^*)$. Taking $\mu^*_n\in (0,+\infty]$ such that $c^*_n=\inf\limits_{\mu>0}\Phi_n(\mu)=\Phi_n(\mu^*_n)$, we have the following claim.
{\\ \it Claim:} There exist $\mu_1>0$ and $\mu_2>0$ such that $\mu_n^* \in [\mu_1,\mu_2]$ for $n$ large enough.

Fix an $\epsilon_0>0$. It follows from Lemma  \ref{lem:basic_phi_psi} $(i)$ and $\inf_{\mu > 0}\Phi(\mu)<\Phi(+\infty)$ that there are $\mu_1< \mu_*$ and $\mu_2>\mu_*$ such that $\Phi(\mu_1)>\Phi(\mu_*)+3\epsilon_0$ and $\Phi(\mu_2)>\Phi(\mu_*)+3\epsilon_0$. Note that $\lambda_n(\mu)$ converges to $\lambda(\mu)$ as $n \rightarrow +\infty$ for any $\mu\in [0,+\infty)$. There exists some integer $N>0$ such that $\Phi_n(\mu^*) \leq \Phi(\mu^*)+\epsilon_0$, $\Phi_n(\mu_1) \geq \Phi(\mu_1)-\epsilon_0$ and $\Phi_n(\mu_2) \geq \Phi(\mu_2)-\epsilon_0$ for all $n\geq N$. We then have $\Phi_n(\mu_1) >\Phi_n(\mu_*)$ and $\Phi_n(\mu_2) >\Phi_n(\mu_*)$
for all $n \geq N$. Hence there exist $\nu_n^1 \in [\mu_1,\mu^*]$ and $\nu_n^2 \in [\mu^*,\mu_2]$  such that $\Phi_n'(\nu_n^1)<0$ and $\Phi_n'(\nu_n^2)>0$ for all $n \geq N$. Lemma \ref{lem:basic_phi_psi}$(iii)$ implies $\mu_n^*< +\infty, ~ \forall n \geq N$. We also deduce that $\mu^*_n \in (\nu_n^1,\nu_n^2) \subset [\mu_1,\mu_2], ~ \forall n \geq N$ due to  Lemma \ref{lem:basic_phi_psi}$(iii)$ and $\Phi_n'(\mu^*_n)=0$. The claim is proved.

Therefore, the desired conclusion can be derived by Lemma \ref{lem:Psi_Phi:con} and the above claim.
\end{proof}
Now we begin to investigate the case where $\inf\limits_{\mu>0} \Phi(\mu)=\Phi(+\infty)$.
\begin{prop}\label{prop:conti:inf=:de}
Assume that $\lambda_n (\mu) \geq \lambda_{n+1} (\mu)\geq \lambda (\mu)$ for all $n\geq 1$ and $\mu \in [0,+\infty)$. If $\inf\limits_{\mu>0} \Phi(\mu)=\Phi(+\infty)$, then $\lim\limits_{n\rightarrow +\infty} \cn=c^*$.
\end{prop}
\begin{proof}
Without loss of generality, we can assume $c^*=0$ (otherwise, $\lambda(\mu)$ changes to $e^{-c^*\mu}\lambda(\mu)$ and $\lambda_n(\mu)$ changes to $e^{-c^*\mu}\lambda_n(\mu)$ for any $n\geq 1$). We deduce that $\Phi_n (\mu) \geq \Phi_{n+1} (\mu) \geq  \Phi (\mu)$ and $c^*_n \geq c^*_{n+1} \geq c^*$ for any $n \geq 1$ because of $\lambda_n (\mu) \geq \lambda_{n+1} (\mu)\geq \lambda (\mu)$. Therefore, $\cn \rightarrow \overline{c}$ for some $\overline{c}\geq 0$ as $n \rightarrow +\infty$. It is sufficient to prove $\overline{c}=0$.  Suppose $\overline{c}>0$, and let $\epsilon_0=\frac{\overline{c}}{3}$. Since $\Phi(+\infty)=0$ and $\lim\limits_{\mu\rightarrow 0^+}\Phi(\mu)=+\infty$, it is easily seen that there exists $\mu_0>0$ such that $\Phi(\mu_0)=\epsilon_0$. Hence, there exists a sufficiently large integer $N > 0$ such that $\Phi_n(\mu_0)\leq \epsilon_0+\Phi(\mu_0)\leq 2 \epsilon_0,~ \forall n \geq N$, which contradicts with $c^*_n=\inf\limits_{\mu>0}\Phi_n(\mu)\geq \overline{c}=3\epsilon_0$. This completes the proof.
\end{proof}

\begin{prop}\label{prop:conti:inf=:in}
Assume that $\lambda_n (\mu) \leq \lambda_{n+1} (\mu) \leq \lambda (\mu)$ for all $n\geq 1$ and $\mu \in [0,+\infty)$ and $\lambda_n'(\mu)$ converges to $\lambda'(\mu)$ as $n \rightarrow +\infty$ for all $\mu\in [0,+\infty)$. If $\inf\limits_{\mu>0} \Phi(\mu)=\Phi(+\infty)$, then $\lim\limits_{n\rightarrow +\infty} \cn=c^*$.
\end{prop}
\begin{proof}
Taking $\mu^*\in (0,+\infty]$ and $\mu^*_n\in (0,+\infty]$ such that $\Phi(\mu^*)=\inf\limits_{\mu>0}\Phi(\mu)$ and $\Phi_n(\mu^*_n)=\inf\limits_{\mu>0}\Phi_n(\mu),~ \forall n \geq 1$, we have $\Phi(\mu^*_n)=\Psi(\mu^*_n)$ due to Proposition $\ref{prop:Phi_Psi}$.
It then follows from $\lambda_n (\mu) \leq \lambda_{n+1} (\mu) \leq  \lambda(\mu)$ that $c^*_n \leq c^*_{n+1} \leq c^*,~ \forall n \geq 1$. Therefore, $\cn \rightarrow \overline{c}$ for some $\overline{c}\leq c^*$ as $n \rightarrow +\infty$.

Next, we prove $c^*=\overline{c}$. Supposing on the contrary that $\overline{c}<c^*$, we have the following claim.
{\\ \it Claim:} $\mu^*_n \rightarrow +\infty$ as $n \rightarrow +\infty$.

Indeed, if the claim is not right, then there exists a subsequence $n_k \geq 1$ with $\lim\limits_{k \rightarrow +\infty} n_k=+\infty$ such that $\omu=\lim\limits_{k\rightarrow +\infty} \mu^*_{n_k}<+\infty$. Lemma \ref{lem:Psi_Phi:con} implies that $\lim\limits_{k \rightarrow +\infty} \Phi_{n_k}(\mu^*_{n_k})=\Phi(\omu)$. Therefore, we have
$$
\overline{c}=\Phi(\omu)\geq \inf_{\mu >0}\Phi(\mu)=c^*>\overline{c}.
$$
This contradiction finishes the proof of the claim.

We now fix any $\mu>0$. There is a large enough number $N= N(\mu)>0$ such that $\mu^*_n>\mu,~ \forall \,n \geq N$ by the above claim. Lemma \ref{lem:basic_phi_psi}$(iv)$ shows that $\Psi_n(\mu) \leq \Psi_n(\mu^*_n),~ \forall\, n\geq N$, and hence,
$$\Psi(\mu)=\lim_{n\rightarrow +\infty} \Psi_n (\mu) \leq \lim_{n\rightarrow +\infty} \Psi_n (\mu^*_n)=\lim_{n\rightarrow +\infty} c_n^*=\overline{c}.$$ We conclude that $c^*= \Psi(+\infty)\leq \overline{c}$, which is a contradiction.
\end{proof}

The we obtain the following consequence.
\begin{thm}\label{thm:conti}
	Assume that $\lambda_n (\mu) \geq \lambda_{n+1} (\mu)\geq \lambda (\mu)$ or $\lambda_n (\mu) \leq \lambda_{n+1} (\mu) \leq \lambda (\mu)$ for all $n\geq 1$ and $\mu \in [0,+\infty)$ and $\lambda_n'(\mu)$ converges to $\lambda'(\mu)$ as $n \rightarrow +\infty$ for all $\mu\in [0,+\infty)$.
	Then $\lim\limits_{n\rightarrow +\infty} c_n^*=c^*$.
\end{thm}

\subsection{The discussion about a limiting case}
Throughout this subsection, choosing $s_0>0$, we let $\Lambda(s,\mu)$ be a function in $C^2( [0,s_0)\times [0,+\infty))$ with the following properties.
\begin{enumerate}[(K1)]
\item $\Lambda(0,0)=1$.
\item $\Lambda_{s}(0,0)>0$.
\item $\ln \Lambda(s,\mu)$ is convex with respect to $\mu\in [0,+\infty)$  for any $s\in [0,s_0)$.
\item $(\ln \Lambda)_{\mu \mu}(0,0)>0$.
\end{enumerate}

For any $\mu \in (0,+\infty)$, we write
 $$
\tPhi(s,\mu):=\frac{\ln \Lambda(s,\mu)}{\mu},
$$
and  for any $\mu \in [0,+\infty)$, we set
$$
\tPsi(s,\mu):=\frac{\Lambda_{\mu}(s,\mu)}{\Lambda(s,\mu)},\quad
H(s,\mu):=\ln \Lambda(s,\mu), \quad
G(s,\mu):=H_{\mu}(s,\mu)\mu- H(s,\mu).
$$

Define
$$c^*(s):=\inf_{\mu>0}\tPhi(s,\mu),~ \forall s \geq 0$$

According to (K1) and (K2), we have $\Lambda(s,0)>1, ~\forall s\geq 0$. Then the sign of $c(s)$ can be discussed by above subsections for all $s>0$.
The following lemma is to show an interesting problem: when $s$ goes to $0$, where will $c^*(s)$ go?
\begin{thm}\label{thm2.6}
There exist a unique $\mu^*(s)$ with $0<\mu^*(s)<+\infty$ and some $p \in (0,s_0)$ such that $\tPhi(s,\mu^*(s))=c^*(s)$ for any $0<s<p$. Moreover, $\mu^*(s)\rightarrow 0$ as $s \rightarrow 0^+$ and $\lim\limits_{s\rightarrow 0^+} c^*(s)=\tPsi(0,0)$.
\end{thm}
\begin{proof}
 By the implicit function theorem, there exist $\mu_1>0$ and continuously differential function $s(\mu)$ on $[0,\mu_1]$ with $s(0)=0$ such that $G(s(\mu),\mu)=0$ due to $G_{s}(0,0)=\Big(\mu H_{\mu s}(s,\mu)-\frac{\Lambda_s(s,\mu)}{\Lambda(s,\mu)}\Big)\Big\vert_{(0,0)}<0$. By virtue of $H_{\mu \mu}(0,0)>0$, there are $\mu_2>0$ and $s_2\in (0,s_0)$ such that $G_{\mu}(s,\mu)=\mu H_{\mu \mu}(s,\mu) >0$ on $ [0,s_2] \times (0,\mu_2]$.  We derive $s'(\mu)=-\frac{G_{u}(s,\mu)}{G(s,\mu)}>0, ~ \forall \mu\in (0,\mu_3]$, where $\mu_3=\min\{\mu_1,\mu_2\}$. Hence, $s(\mu)$ is strictly increasing and continuous on $[0,\mu_3]$. Letting $s_3=s(\mu_3)$, we obtain that $s(\mu)$ admits a continuous inverse function $\mu^*(s)$ on $[0,s_3]$ and $G(s,\mu^*(s))=0,~ \forall s \in [0,s_3]$. It then follows that $c^*(s)=\tPhi(s,\mu^*(s))=\tPsi(s,\mu^*(s)), ~ \forall s\in (0,s_3]$. Therefore, we conclude that $\lim\limits_{s\rightarrow 0^+} \mu^* (s)=0$ and $\lim\limits_{s\rightarrow 0^+} c^*(s)=\lim\limits_{s\rightarrow 0^+} \tPsi(s,\mu^*(s))=\tPsi(0,0)$. This completes the proof.
\end{proof}
\begin{remark}
In the proof of Theorem \ref{thm2.6}, the implicit function theorem cannot be applied for $\mu$ at $(0,0)$ due to $G_{\mu}(0,0)=0$.
\end{remark}

\section{An application to cellular neural networks}
In this section, we investigate propagation phenomena and some properties of spreading speeds for cellular neural networks
\begin{align}\label{3.1}
\frac{dx_i(t)}{dt} =-x_i(t)+\alpha f(x_{i-1}(t))+af(x_{i}(t))+\beta f(x_{i+1}(t)),\,\,i\in\mathbb{Z},
\end{align}
where the output function $f(u)=\frac{1}{2}(|u+1|-|u-1|)$ and the parameters $\alpha,a,\beta$ are nonnegative.


Now we give the following assumptions.

(H) The nonnegative parameters $\alpha$, $a$ and $\beta$ satisfy $\alpha+\beta>0$ and $\alpha+a+\beta>1$.

According to (H), \eqref{3.1} has three equilibria 0 and $\pm K$, where $K=\alpha+a+\beta.$


\subsection{Existence of spreading speeds}

 Let $Q_t$ be the solution map at time $t\geq0$ of system \eqref{3.1}, that is,
$$
Q_t(x^0)=x(t,x^0),\quad \forall\,\,x^0=\{x^0_{i}\}_{i\in\mathbb{Z}}\in\mathcal{X}_{K},
$$
where $\mathcal{X}_K=\{\varphi=\{\varphi_i\}_{i\in\mathbb{Z}}\,|\,\varphi_i\in [0,K],\,i\in\mathbb{Z}\}$.
We can easily check that  $Q:=Q_1$ satisfy all hypotheses (A1)--(A6) in \cite{liang2010spreading}. Thus,  there exist $c^*_+$ and $c^*_{-}$ are the rightward and leftward spreading speeds of $Q$, respectively.

Firstly, we estimate the rightward spreading speeds. Therefore, we consider the linearized equation of \eqref{3.1} at the zero solution, i.e.,
\begin{align}\label{3.2}
\frac{dx_i(t)}{dt} =-x_i(t)+\alpha x_{i-1}(t)+ax_{i}(t)+\beta x_{i+1}(t),\,\,i\in\mathbb{Z}.
\end{align}
Let $\{M_t\}_{t\geq0}$ be the solution semiflow associated with \eqref{3.2}. Thus, for each $t>0$, the map $M_t$ satisfies the assumptions (C1)--(C5) in \cite{liang2010spreading}. Notice that $f'(0)u\geq f(u)$ for $u\in[0,K]$, where $f'(0)=1$. By the comparison theorem, we have $Q_t(x^0)\leq M_t(x^0),\,\forall\, x^0\in\mathcal{X}_K,\,\, t\geq0.$ On the other hand, for any $\epsilon>0$, there exists $\delta>0$ such that, for $x^0\in \mathcal{X}_K$ with $x^0<\delta$, we can obtain $Q_t(x^0)\geq M^\epsilon_t(x^0)$ for all $t\in [0,1]$, where $ M^\epsilon_t$ is the solution semiflow of
\begin{align}\label{3.02}
\frac{dx_i(t)}{dt} =-x_i(t)+(1-\epsilon)\alpha x_{i-1}(t)+(1-\epsilon)ax_{i}(t)+(1-\epsilon)\beta x_{i+1}(t),\,\,i\in\mathbb{Z}.
\end{align}

Taking $x_i(t)=e^{-\mu i}v(t)$ is a solution \eqref{3.2}, $v(t)$ satisfies the following differential equation:
\begin{align}\label{3.3}
\frac{dv(t)}{dt} =(a-1+\alpha e^{\mu}+\beta e^{-\mu})v(t).
\end{align}

Letting \begin{align*}
B_\mu^t(v_0):=M_t[v_0 e^{-\mu i}](0)=v(t,v_0),\,\forall\,\, v(0)=v_0\in [0,\infty),
\end{align*}
it follows that $B_\mu^t$ is the solution map at time $t$ of equation \eqref{3.3} and
$$B_\mu^t(v_0)=e^{(a-1+\alpha e^{\mu}+\beta e^{-\mu})t}v_0,\,\forall\,\, v_0\in [0,\infty).$$
Thus, for any $\mu\geq0$, $B_\mu:=B_\mu^1$ is a compact and strongly positive linear operator on $[0,\infty)$, i.e., (C6) in \cite{liang2010spreading} holds. It is obvious to see that
$$
\lambda(\mu)=e^{a-1+\alpha e^{\mu}+\beta e^{-\mu}}
$$
is the principal eigenvalue of $B_\mu$ for any $\mu\geq0$ and $\lambda(0)=e^{a-1+\alpha+\beta}>1$.

Let
$$
\Phi(\mu)=\frac{\ln \lambda(\mu)}{\mu}=\frac{h(\mu)}{\mu},~\mu\neq0\quad\mbox{and}\quad
\Psi(\mu)=\frac{\lambda'(\mu)}{\lambda(\mu)},~\mu\in(-\infty,+\infty),
$$
where $h(\mu)=a-1+\alpha e^{\mu}+\beta e^{-\mu}$. It is obvious that Lemma \ref{lem:basic_phi_psi} holds. We denote
$$
\Phi(+\infty):=\lim\limits_{\mu\rightarrow +\infty} \Phi(\mu)\qquad
\Phi^+(\mu)=\Phi(\mu),\qquad
\Phi^-(\mu)=-\Phi(-\mu),
$$
$$
\Psi(+\infty):=\lim\limits_{\mu\rightarrow +\infty} \Psi(\mu)\qquad
\Psi^+(\mu)=\Psi(\mu),\qquad
\Psi^-(\mu)=-\Psi(-\mu).
$$
According to Proposition 3.9 and Theorem 3.10 in \cite{liang2007asymptotic} and Lemma 4.6 in \cite{ding2015principal} (including that $\inf_{\mu>0}\Phi^\pm(\mu)=\Phi^\pm(+\infty)$ ), we have
\begin{align}\label{3.4}
c^*_+=\inf_{\mu>0}\Phi^+(\mu)=\inf_{\mu>0}\frac{\ln \lambda(\mu)}{\mu}=\inf_{\mu>0}\frac{a-1+\alpha e^{\mu}+\beta e^{-\mu}}{\mu},
\end{align}

Similarly, it follows that the left spreading speed
\begin{align}\label{3.5}
c^*_-=\inf_{\mu>0}\Phi^-(\mu)=\inf_{\mu>0}\frac{\ln \lambda(-\mu)}{\mu}=\inf_{\mu>0}\frac{a-1+\alpha e^{-\mu}+\beta e^{\mu}}{\mu},
\end{align}

Now we want to prove $c_{+}^*+c_{-}^*>0$.
\begin{prop}\label{prop:both>0}
	Assume that {\rm (H)} hold. Then $c_{+}^*+c_{-}^*>0$.
\end{prop}
\begin{proof}
It follows from Lemma \ref{lem:basic_phi_psi} that there exists $\mu^*_+\in(0,+\infty]$ and $\mu^*_-\in(0,+\infty]$ such that $c^*_+=\Phi^+(\mu^*_+)$ and $c^*_-=\Phi^-(\mu^*_-)$. By Proposition \ref{prop:Phi_Psi} we have $\Phi^+(\mu^*_+)=\Psi^+(\mu^*_+)=\Psi(\mu^*_+)$ and $\Phi^-(\mu^*_-)=\Psi^-(\mu^*_-)=-\Psi(-\mu^*_-)$. Since $\Psi(\mu)$ is strictly increasing in $\R$, we conclude $c^*_++c^*_-=\Psi(\mu^*_+)-\Psi(-\mu^*_-)>0$.
\end{proof}
As a direct result of Theorems 3.4 in \cite{liang2010spreading} and Theorem 2.12 in \cite{ding2015principal}, we have the following conclusions.

\begin{thm}\label{thm3.1}Assume that {\rm (H)} holds. Let $x(t)$ be a solution of \eqref{3.1} with the initial condition $x^0\in \mathcal{X}_K$. Then $c^*_+$ and $c^*_{-}$  defined by \eqref{3.4} and \eqref{3.5} are the rightward and leftward
spreading speeds of $Q_1$, respectively, such that the following statements are valid:
\begin{itemize}
  \item [(i)]  For any $c > c^*_+$ and $c' > c^*_{-}$, if $x^0\in \mathcal{X}_K$ with $x^0_i  = 0$ for $i$ outside a bounded interval, then
$\lim\limits_{t\rightarrow\infty,i\geq ct} x_i (t) = 0$  and $\lim\limits_{t\rightarrow\infty,i\leq -c't} x_i (t) = 0$.

  \item [(ii)]  For any $c < c^*_+$ and $c' < c^*_{-}$, if $x^0\in \mathcal{X}_K\backslash\{0\}$, then $\lim\limits_{t\rightarrow\infty,-c't\leq i\leq ct} x_i (t) = K$.
\end{itemize}
\end{thm}



\subsection{The sign of spreading speeds}
In this subsection, we investigate the sign of spreading speeds for CNNs by using the results in Section 2.
\begin{prop}\label{prop3.3} Assume that {\rm (H)} hold. Then the following statements hold.

\begin{itemize}
\item [(i)] If $\alpha>\beta$, then $c^*_+>c^*_-$ and $c^*_+>0$.
  \item [(ii)] If $\alpha=\beta>0$, then $c^*_+=c^*_->0.$
  \item [(iii)] If $\alpha<\beta$, then $c^*_+<c^*_-$ and $c^*_->0$.
  \end{itemize}
\end{prop}
\begin{proof} It is easily checked that (L1)-(L3) hold.
If $\alpha>\beta$, then
\begin{align*}c^*_+-c^*_-=&\inf\limits_{\mu>0}\Phi^+(\mu)-\inf\limits_{\mu>0}\Phi^-(\mu)\\
\geq&\inf\limits_{\mu>0}[\Phi^+(\mu)-\Phi^-(\mu)]\\
=&(\alpha-\beta)\inf\limits_{\mu>0}\frac{(e^{\mu}-e^{-\mu})}{\mu}>0.
\end{align*}
Therefore $c^*_+>0$ from Proposition \ref{prop:both>0}.
By the similar way, we can prove the case where $\alpha<\beta$.
If $\alpha=\beta$, it is obvious that $\Phi^+(\mu)=\Phi^-(\mu)$, which implies that $c^*_+=c^*_-$.
 This completes the proof.
\end{proof}

\begin{prop}\label{prop:CNN:threhold:beta>alpha>0} Assume that {\rm (H)} holds. In addition, $\beta>\alpha>0$, then the following conclusions hold:
\begin{itemize}
  \item [(i)] If $2\alpha^{\frac{1}{2}}\beta^{\frac{1}{2}}+a-1>0$, then  $c^*_+>0$.
  \item [(ii)] If $2\alpha^{\frac{1}{2}}\beta^{\frac{1}{2}}+a-1=0$, then $c^*_+=0$.
  \item [(iii)] If $2\alpha^{\frac{1}{2}}\beta^{\frac{1}{2}}+a-1<0$, then  $c^*_+<0$.
\end{itemize}
\end{prop}
\begin{proof} It is obvious to see that (L1)--(L3) and (L4$'$) hold. Notice that that$$\min\limits_{\mu\in(0,+\infty)}h(\mu)=h(\mu_+^0)=2\alpha^{\frac{1}{2}}\beta^{\frac{1}{2}}+a-1,$$ where $\mu_+^0=\frac{1}{2}\ln\frac{\beta}{\alpha}>0$.
Then the conclusion can be easily checked by Corollary \ref{cor:sign}.
\end{proof}

\begin{prop}\label{prop:CNN:threhold:beta>alpha=0}
 Assume that {\rm (H)} holds. In addition, $a+\beta>1$ and  $\beta>\alpha=0$, then we can obtain the following conclusions.
\begin{itemize}
  \item [(i)] If $0\leq a<1$, then $c^*_+<0$.
  \item [(ii)] If $a\geq1$, then $c^*_+=0$.
\end{itemize}
\end{prop}
\begin{proof} It is obvious that (L1)--(L3) hold under the condition that  $a+\beta>1$ and  $\beta>\alpha=0$. Since $\lim_{\mu\rightarrow+\infty}\Phi^+(\mu)=\lim_{\mu\rightarrow+\infty}\frac{h(\mu)}{\mu}=0$, (L4) also holds. Note that $\lambda^+(+\infty)=e^{a-1}$. Thus, when $0\leq a< 1$,
 we have $\lambda^+(+\infty)< 1$, which implies that $c^*_+<0$ by Theorem \ref{thm1.1}. When $a\geq 1$, then $\lambda^+(+\infty)\geq 1$ and $c^*_+=0$.
\end{proof}
\begin{remark}
For the cases $\alpha>\beta>0$ and $\alpha>\beta=0$, we can obtain the corresponding results similar to Propositions \ref{prop3.3}--\ref{prop:CNN:threhold:beta>alpha=0} (see, Tables 1 and 2).
\end{remark}

\subsection{Continuity of spreading speeds}
In this subsection, we consider CNNs with the variable template $[\alpha_n,a_n,\beta_n]$ as follows:
\begin{align}\label{3.6}
\frac{dx_i(t)}{dt} =-x_i(t)+\alpha_n f(x_{i-1}(t))+a_nf(x_{i}(t))+\beta_n f(x_{i+1}(t)),\,\,i\in\mathbb{Z},
\end{align}
where the nonnegative parameters $\alpha_n,a_n,\beta_n$ $(n\in\mathbb{N})$ satisfy
\begin{itemize}
  \item [(P)] $
\lim\limits_{n\rightarrow+\infty}\alpha_n=\alpha, \lim\limits_{n\rightarrow+\infty}a_n=a\,\,\mbox{and}\,\,\lim\limits_{n\rightarrow+\infty}\beta_n=\beta,
$
where $\alpha,a,\beta$ satisfy the assumption (H).
\end{itemize}
We mainly investigate the relation between the spreading speeds of CNNs with the template $[\alpha_n,a_n,\beta_n]$ and with the template $[\alpha,a,\beta]$.

According to the assumption (P), there exists a sufficiently large number $N_0\in\mathbb{N}$ such that $\alpha_n,a_n,\beta_n$ $(n\in\mathbb{N})$ also satisfy (H) for $n>N_0$. Thus it follows from Theorem \ref{thm3.1} that, for any $n>N_0$, \eqref{3.6} admits the right and left spreading speeds
$$
{c_n}^*_+=\inf_{\mu>0}\Phi_n^+(\mu)=\inf_{\mu>0}\frac{\ln \lambda_n(\mu)}{\mu}=\inf_{\mu>0}\frac{a_n-1+\alpha_n e^{\mu}+\beta_n e^{-\mu}}{\mu}
$$
and
$$
{c_n}^*_-=\inf_{\mu>0}\Phi_n^-(\mu)=\inf_{\mu>0}\frac{\ln \lambda_n(-\mu)}{\mu}=\inf_{\mu>0}\frac{a_n-1+\alpha_n e^{-\mu}+\beta_n e^{\mu}}{\mu}.
$$

\begin{prop}\label{prop3.5}
Assume that {\rm (H) and (P)} hold. Then $\lim\limits_{n\rightarrow+\infty}{c_n}^*_{\pm}=c^*_{\pm}$.
\end{prop}
\begin{proof}
 We will verify $\lim\limits_{n\rightarrow+\infty}{c_n}^*_+=c^*_+$ by using Theorem \ref{thm:conti}. The other case can be derived by the same method.

Define $\bar{\alpha}_n=\max\limits_{k\geq n}\{\alpha,\alpha_k\}$ and $\underline{\alpha}_n=\min\limits_{k\geq n}\{\alpha,\alpha_k\}$. It is not hard to verify that $\{\bar{\alpha}_n\}_{n\in\mathbb{N}}$ and $\{\underline{\alpha}_n\}_{n\in\mathbb{N}}$ are nonincreasing and nondecreasing sequences, respectively. Moreover, $\bar{\alpha}_n\geq \alpha\geq\underline{\alpha}_n, \forall \,n \geq 1$. According to the assumption (P), we can obtian that
$$
\lim\limits_{n\rightarrow+\infty}\bar{\alpha}_n=\alpha\,\,\mbox{and}\,\,\lim\limits_{n\rightarrow+\infty}\underline{\alpha}_n=\alpha.
$$
Similarly, for any $n \geq 1$, define
$$\bar{a}_n=\max\limits_{k\geq n}\{a,a_k\},~ \underline{a}_n=\min\limits_{k\geq n}\{a,a_k\},~ \bar{\beta}_n=\max\limits_{k\geq n}\{\beta,\beta_k\},~\underline{\beta}_n=\min\limits_{k\geq n}\{\beta,\beta_k\}
$$ with
$$
\lim\limits_{n\rightarrow+\infty}\bar{a}_n=a,\,\,\lim\limits_{n\rightarrow+\infty}\underline{a}_n=a,\,\,
\lim\limits_{n\rightarrow+\infty}\bar{\beta}_n=\beta\,\,\mbox{and}\,\,\lim\limits_{n\rightarrow+\infty}\underline{\beta}_n=\beta
$$
and
$$\bar{a}_n\geq a\geq\underline{a}_n,\,\,\bar{\beta}_n\geq \beta\geq\underline{\beta}_n.$$
According to (P), there exists sufficiently a large number $N_1\in\mathbb{N}$ such that $$\bar{\alpha}_n+\bar{a}_n+\bar{\beta}_n>1\,\, \mbox{and}\,\, \underline{\alpha}_n+\underline{a}_n+\underline{\beta}_n>1$$ for any $n\geq N_1$. Thus,  for any $n>N_0$, \eqref{3.6} with the templates $[\bar{\alpha}_n,\bar{a}_n,\bar{\beta}_n]$ and $[\underline{\alpha}_n,\underline{a}_n,\underline{\beta}_n]$ admits the right spreading speed ${\overline{c}_n}^*_+$ and ${\underline{c}_n}^*_+$, respectively. In view of Lemma 2.9 in \cite{liang2007asymptotic}, we have
\begin{align}\label{3.7}
{\underline{c}_n}^*_+\leq {c_n}^*_+\leq{\overline{c}_n}^*_+
\end{align}
 for all $n>\max\{N_0,N_1\}$.

On the other hand, we can verify that $\bar{\lambda}_n(\mu)$ and $\underline{\lambda}_n(\mu)$ corresponding to the definition of ${\lambda}_n(\mu)$ is nonincreasing and nondecreasing on $n\in\mathbb{N}$.  Moreover, $\underline{\lambda}_n(\mu)\leq \lambda(\mu)\leq \overline{\lambda}_n(\mu)$ for any $n\in\mathbb{N}$ and $\lim\limits_{n\rightarrow+\infty}\underline{\lambda}_n(\mu)=\lim\limits_{n\rightarrow+\infty}\overline{\lambda}_n(\mu)=\lambda(\mu)$  for any closed set on $(0,+\infty)$. According to Theorem \ref{thm:conti}, we can obtain that
\begin{align}\label{3.8}
\lim\limits_{n\rightarrow+\infty}{\underline{c}_n}^*_+=\lim\limits_{n\rightarrow+\infty}{\overline{c}_n}^*_+= {c}^*_+.
\end{align}
Thus, it follows from \eqref{3.7} and \eqref{3.8} that
$$
\lim\limits_{n\rightarrow+\infty}{{c}_n}^*_+= {c}^*_+.
$$

\end{proof}
\begin{remark}
	Define the sets
	$$I=\{(\alpha,a,\beta)\,|\,\alpha,a,\beta\geq0,\,\alpha+\beta>0, \alpha+a+\beta>1\},$$
	$$
	U=\{(\alpha,a,\beta)\,|\,\alpha=0,a\geq 1,\beta>0\},
	$$
	and
	$$
	V=\{(\alpha,a,\beta)\,|\,\alpha>0,a\geq 1,\beta=0\}.
	$$
	It is worth pointing out that when $[\alpha_n,a_n,\beta_n] \in I\setminus (U\cup V)$ for any $n \geq 1$ and  $[\alpha,a,\beta]\in U$ and {\rm (P)} holds, an interesting phenomenon occurs, that is, $\Phi_n^+(+\infty)=+\infty$ while $\Phi^+(+\infty)=0$.
\end{remark}

\subsection{Discussion about the limiting cases}
In this subsection, we estimate the spreading speed of the limiting cases for
\begin{align}\label{equ:evo:a_alpha_beta_s}
\frac{dx_i(t)}{dt} =-x_i(t)+\alpha(s)f(x_{i-1}(t))+a(s)f(x_i(t))+\beta(s) f(x_{i+1}(t)),\,\,i\in\mathbb{Z},
\end{align}
where the nonnegative parameters $\alpha(s)$, $a(s)$ and $\beta(s)$
 satisfy the following assumption:
\begin{itemize}
\item [(S1)]  the continuous functions $\alpha(s)$, $a(s)$ and $\beta(s)$ are strictly increasing on $s\in [0,+\infty)$.
  \item [(S2)] $
\lim\limits_{s\rightarrow0^+}\alpha(s)=\alpha, \lim\limits_{s\rightarrow0^+}a(s)=a\,\,\mbox{and}\,\,\lim\limits_{s\rightarrow0^+}\beta(s)=\beta,
$
where $\alpha,a,\beta$ satisfy the following condition (H$'$): $\alpha,a,\beta\geq0$, $\alpha+\beta>0$ and $\alpha+a+\beta=1$.
\end{itemize}

Notice that Theorem \ref{thm3.1} does not hold for $s=0$. According to (S1) and (S2), it is easily seen that
$$\alpha(s)+a(s)+\beta(s)>1\quad\mbox{and}\quad a(s)+\beta(s)>0,~ \forall s >0.$$ Therefore, it follows from Theorem \ref{thm3.1} that for any $s>0$, \eqref{equ:evo:a_alpha_beta_s} admits the right and left spreading speeds
$$
{c(s)}^*_+=\inf_{\mu>0}\Phi^+(s,\mu)
$$
and
$$
{c(s)}^*_-=\inf_{\mu>0}\Phi^-(s,\mu),
$$
where
$$
\Phi^\pm(s,\mu)=\frac{\ln \Lambda^\pm(s,\mu)}{\mu}
$$
and
$$\Psi^\pm(s,\mu)=\frac{\Lambda_{\mu}^\pm(s,\mu)}{\Lambda^\pm(s,\mu)}=
\pm\alpha(s)e^{\pm\mu}\mp\beta(s)e^{\mp\mu}$$
That is, we investigate that where $c_\pm^*(s)$ will go, when $s$ goes to $0$. It is easily verified that $\Lambda^+(s,\mu)$ and $\Lambda^-(s,\mu)$ satisfy (K1)--(K4). The following conclusions hold from Theorem \ref{thm2.6}.
\begin{prop}\label{prop3.7}Assume that (S1) and (S2) hold. Then $\lim\limits_{s\rightarrow 0^+}{c(s)}^*_+=\alpha-\beta$ and $\lim\limits_{s\rightarrow 0^+}{c(s)}^*_-=-\alpha+\beta$.
\end{prop}

Now we consider a special version of \eqref{equ:evo:a_alpha_beta_s} as follows:
\begin{align}\label{equ:evo:a_alpha_beta_s:special}
\frac{dx_i(t)}{dt} =-x_i(t)+\alpha f(x_{i-1}(t))+af(x_{i}(t))+(\beta+s) f(x_{i+1}(t)),\,\,i\in\mathbb{Z},
\end{align}
where the nonnegative parameters $\alpha,a,\beta$ satisfy $\alpha+a+\beta=1$ and $\alpha+\beta>0$. It is obvious that (S1) and (S2) hold. Hence, it holds that
$\lim\limits_{s\rightarrow 0^+}{c(s)}^*_+=\alpha-\beta$ and $\lim\limits_{s\rightarrow 0^+}{c(s)}^*_-=-\alpha+\beta.$
\begin{remark}
From the above analysis, it is easy to see that $c(s)^*_-<0$ when $\alpha>\beta+s$ and $s$ is small enough.
\end{remark}
\subsection{Numerical analysis}
In this section, we compute the spreading speeds by numerical method. We first present the simulation under some parameters, which has been listed in Table \ref{tab:c}, and then compute the spreading speeds by the method in \cite[Section 4.2]{lutscher2008density}. Next, we will give the spreading speeds computed by the the formula \eqref{3.4} and \eqref{3.5}, which are approximate with the previous simulation results. We also show the sign of the spreading speed from Tables \ref{tab:alpha>0&beta>0} and \ref{tab:alpha||beta=0}, which admits with the above two results.
\begin{table}[H]
	\centering
	\begin{tabular}{|c|c|c|c|c|c|c|}
		\hline
		Parameters	&$c^*_-(S)$  & $c^*_+(S)$  & $c^*_-(\Phi)$ & $c^*_+(\Phi)$ & Sign of $c^*_-$ & Sign of $c^*_+$  \\
		\hline
		$\alpha$=0.5, $a$=1, $\beta$=0.5 & 1.43 & 1.43 & 1.51 & 1.51&positive& positive\\
		\hline		
		$\alpha$=0.05, $a$=0.5, $\beta$=0.5 & 0.69 & -0.23 & 0.70 & -0.23&positive& negative\\
		\hline
		$\alpha$=0.125, $a$=0.5, $\beta$=0.5 & 0.78 & -0.01 & 0.80 & 0.00&positive & zero\\
		\hline
		$\alpha$=0, $a$=1, $\beta$=0.5 & 1.31 & 0.00 & 1.36 & 0.00&positive& zero\\
		\hline
		$\alpha$=0, $a$=0.55, $\beta$=0.5 & 0.73 & -0.30 & 0.74 & -0.29&positive& negative \\
         \hline
	\end{tabular}
	\caption{$c^*_\pm(S)$, $c^*_\pm(\Phi)$ represent the spreading speeds computed by simulation(see, e.g., \cite[Section 4.2]{lutscher2008density}) as well as the formula \eqref{3.4} and \eqref{3.5}. Sign of $c^*_\pm$ are derived from Tables \ref{tab:alpha>0&beta>0} and \ref{tab:alpha||beta=0}. \label{tab:c}}
\end{table}
We now give more analysis for the case where $\alpha=0.5$, $a=1$, $\beta=0.5$, the case where $\alpha=0$, $a=1$, $\beta=0.5$ and the case where $\alpha=0$, $a=0.55$, $\beta=0.5$.
\begin{figure}[H]
	\subfigure{
	\centering
	\includegraphics[width=2in]{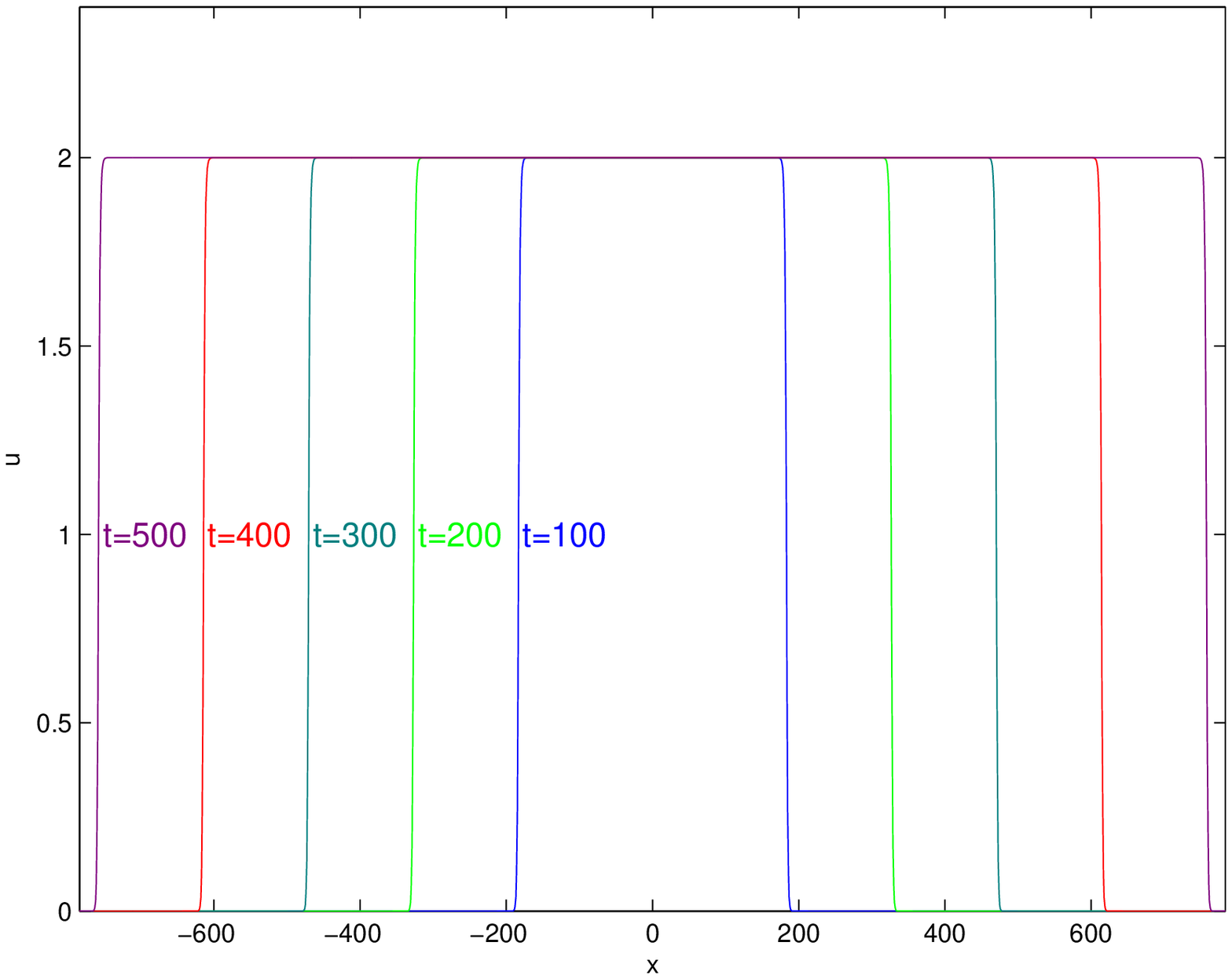}}
	\subfigure{
	\centering
	\includegraphics[width=2in]{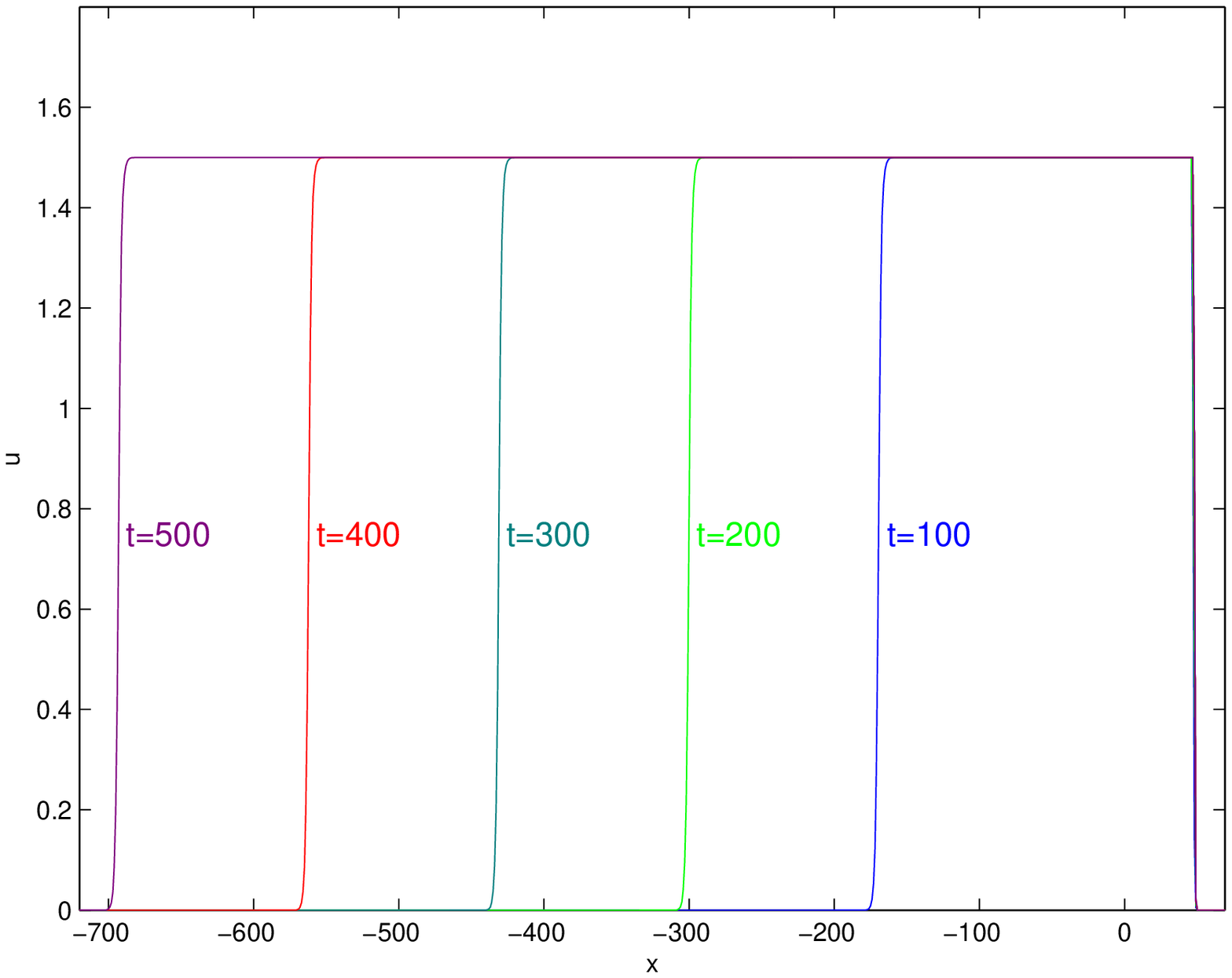}}
	\hfill
	\subfigure{
	\centering
	\includegraphics[width=2in]{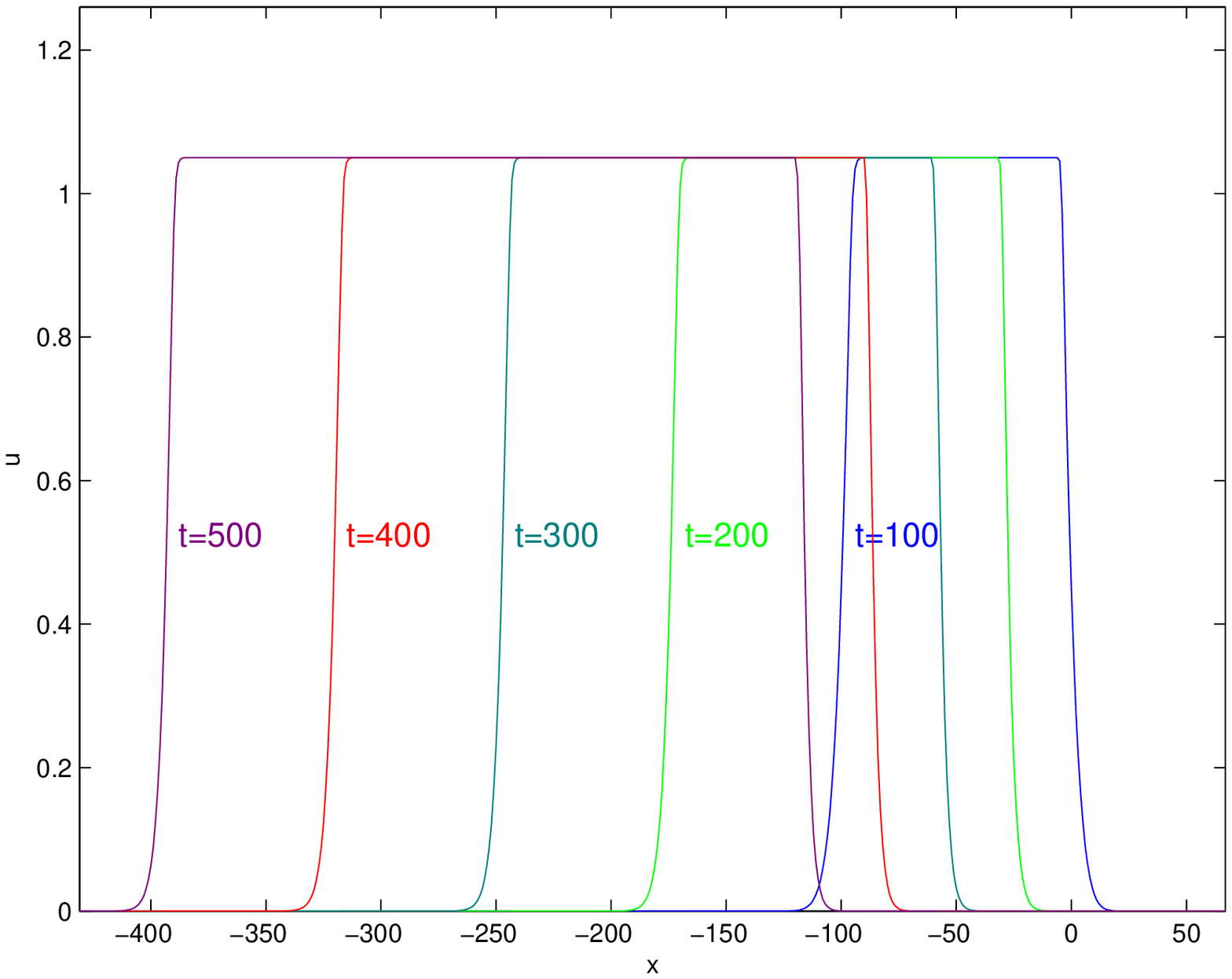}}
\caption{\label{figure:R0LR}  \small The parameters of left figure is $\alpha=0.5$, $a=1$, $\beta=0.5$, the simulation results imply that the signal will transfer to both sides. The parameters of middle figure is $\alpha=0$, $a=1$, $\beta=0.5$, the simulation results imply that the signal will transfer to left side and stop on the right side. The parameters of right figure is $\alpha=0$, $a=0.55$, $\beta=0.5$, the simulation results imply that the signal will transfer to left side and diminish on right side.}
\end{figure}
\begin{figure}[H]
	\centering
	\subfigure{
		\centering
		\includegraphics[width=3in]{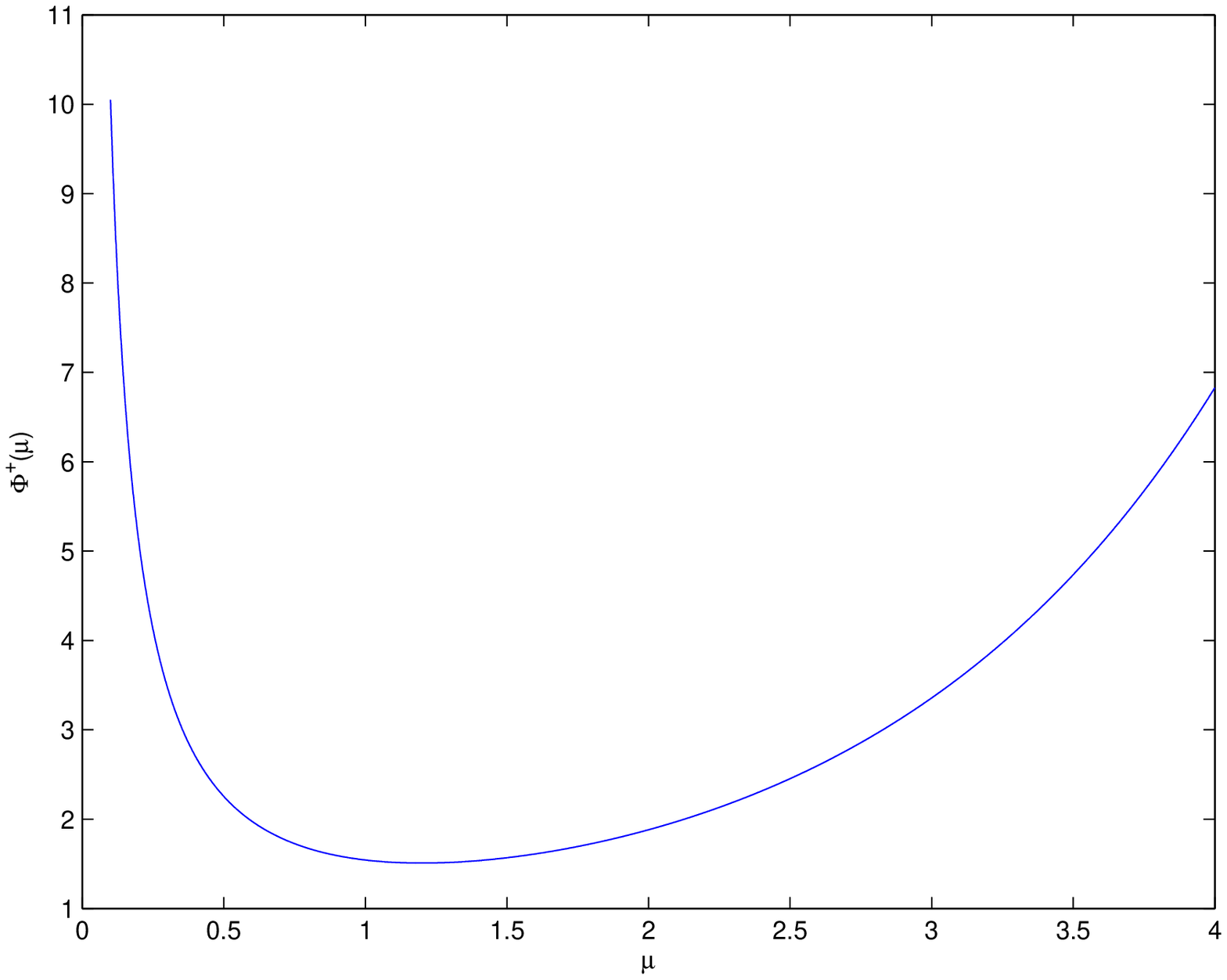}}
	\hfill
	\subfigure{\centering
		\includegraphics[width=3in]{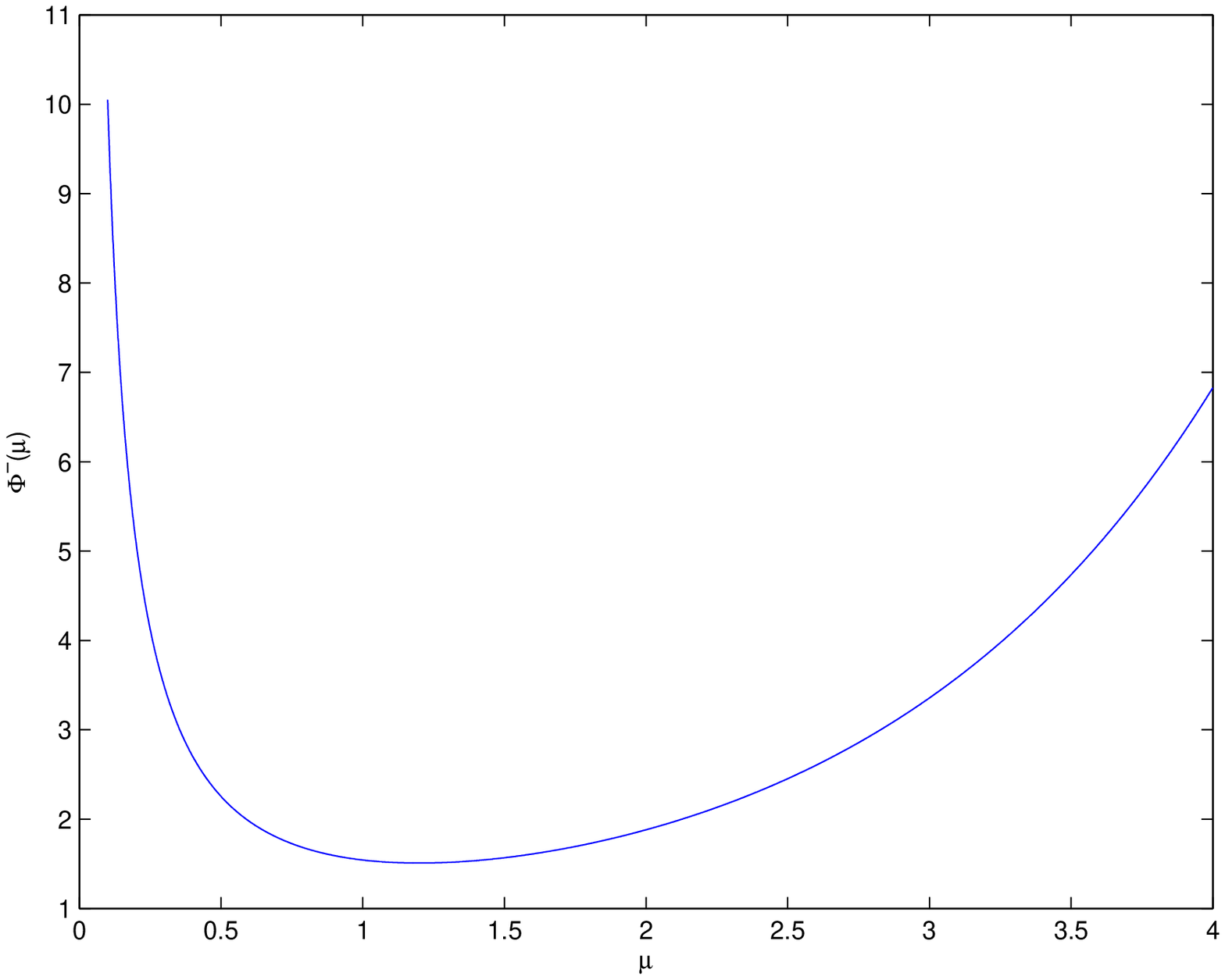}}
	\caption{  \small  $\alpha=0.5$, $a=1$, $\beta=0.5$, both of the infimum of $\Phi^-(\mu)$ and $\Phi^+(\mu)$ are positive and can be attained at some finite value.}
\end{figure}

\begin{figure}[H]
	\centering
	\subfigure{
		\centering
		\includegraphics[width=3in]{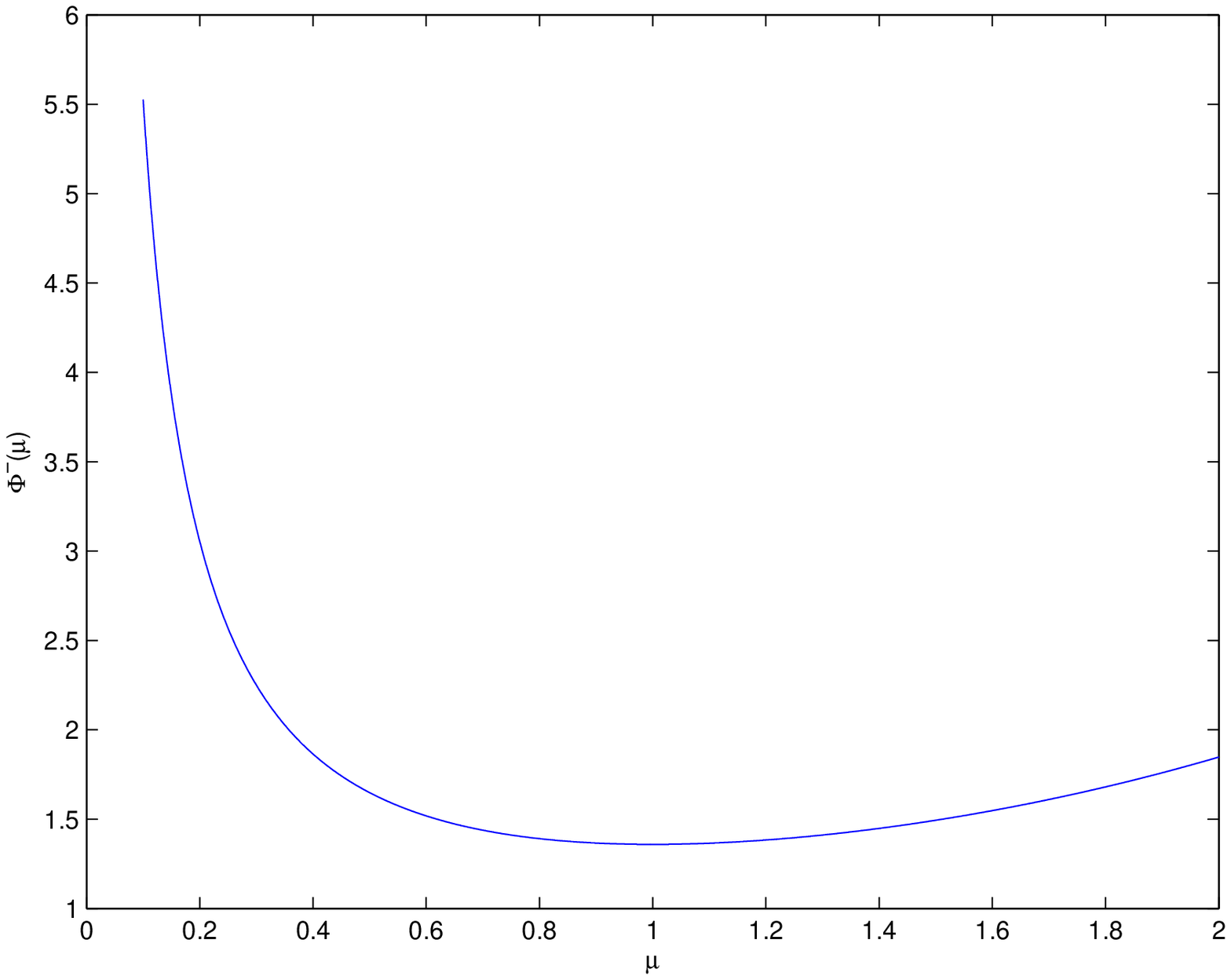}}
	\hfill
	\subfigure{\centering
		\includegraphics[width=3in]{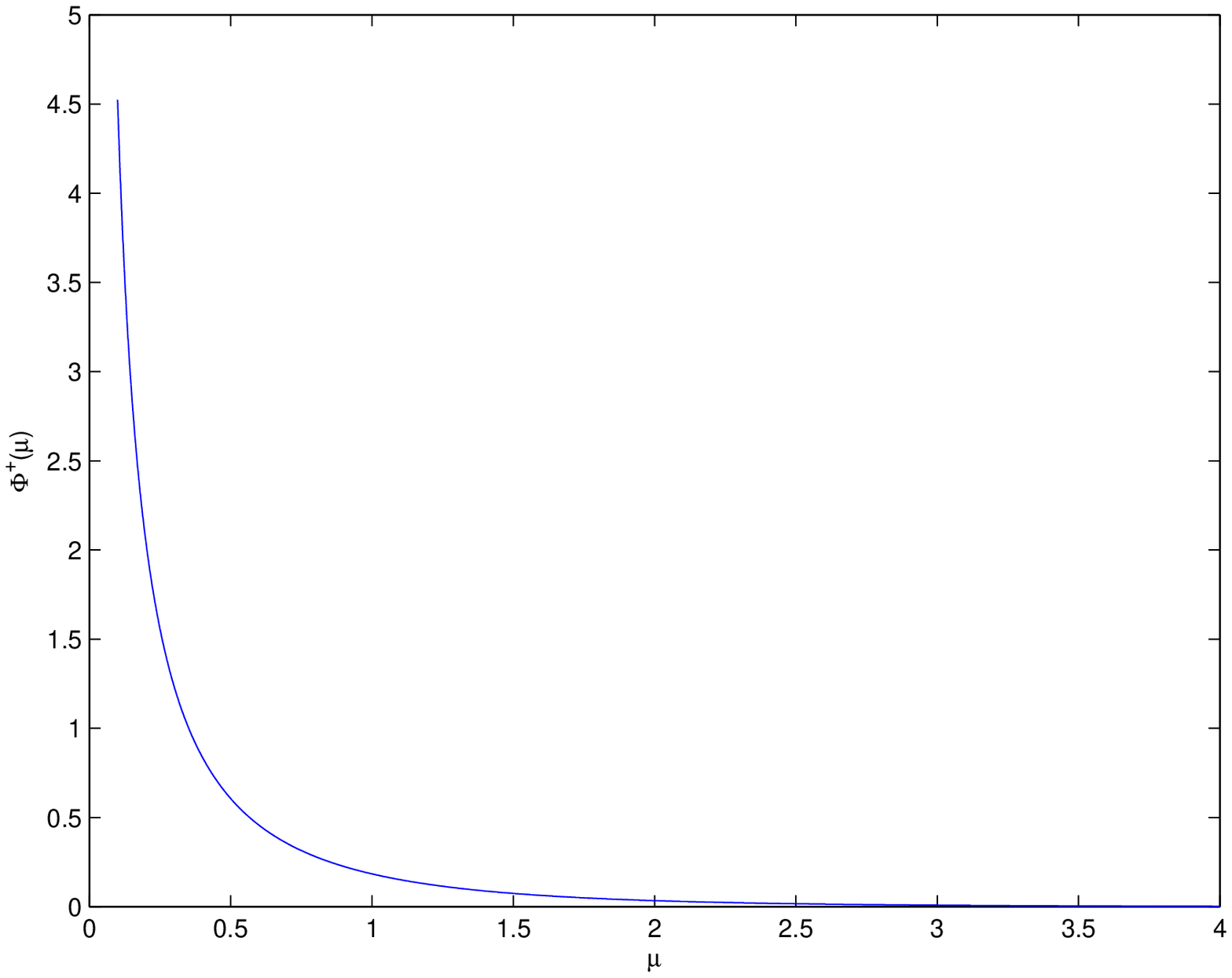}}
\caption{  \small  $\alpha=0$, $a=1$, $\beta=0.5$, the infimum of $\Phi^-(\mu)$ is a positive number which can be attained at some finite value while the infimum of $\Phi^+(\mu)$ is zero which is attained at infinite value.}
\end{figure}
\begin{figure}[H]
	\centering
	\subfigure{
		\centering
		\includegraphics[width=3in]{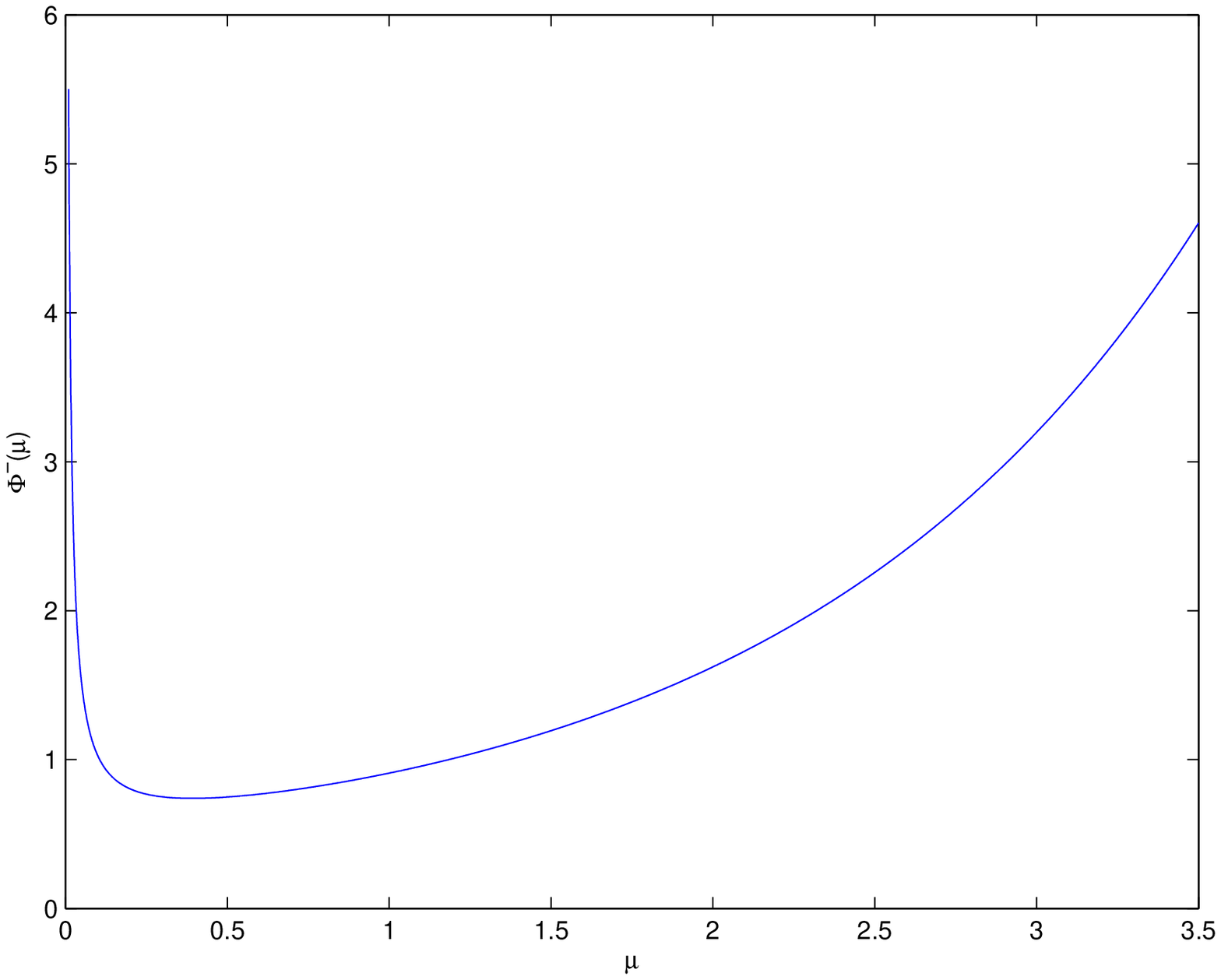}}
	\hfill
	\subfigure{\centering
		\includegraphics[width=3in]{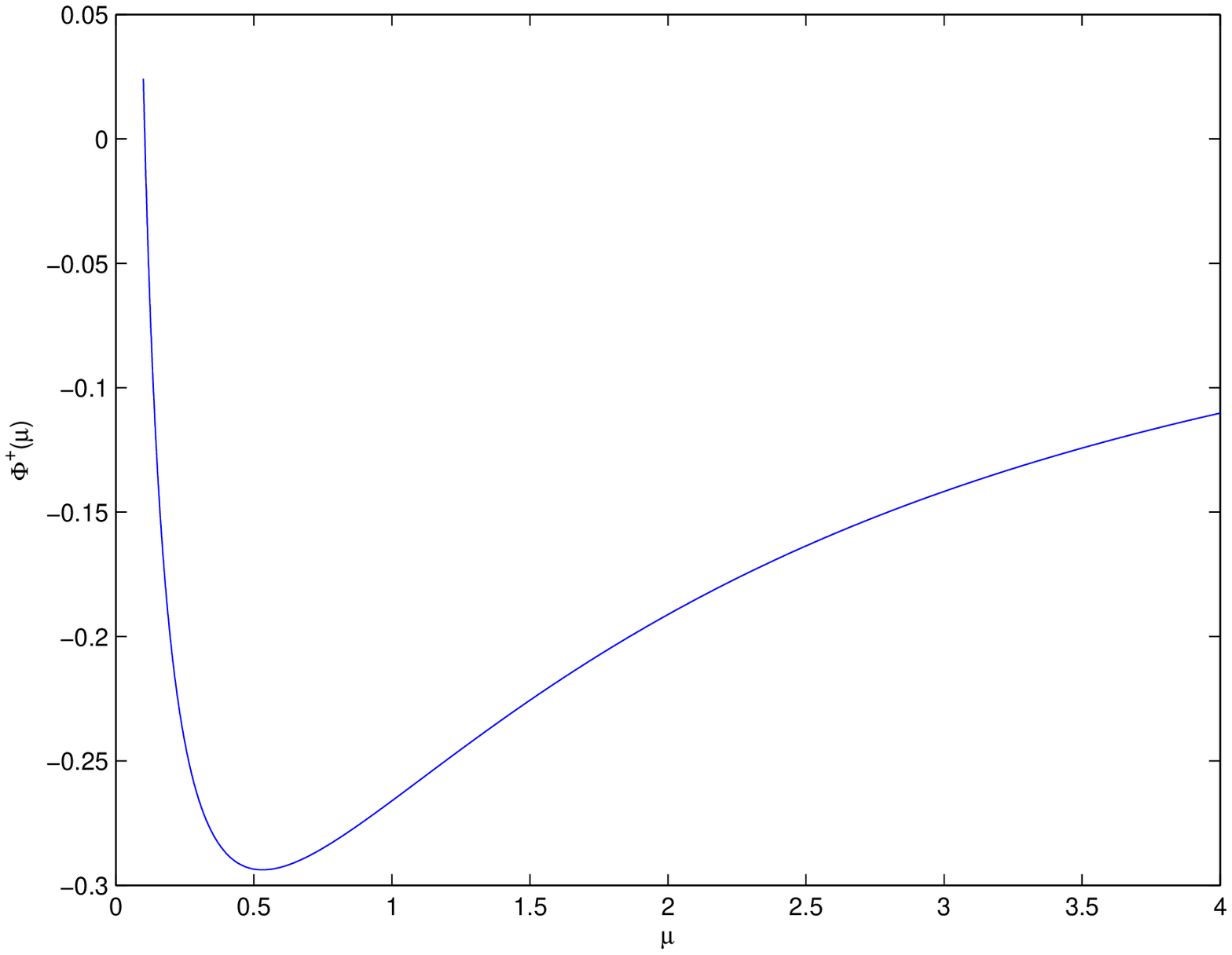}}
	\caption{  \small  $\alpha=0$, $a=0.55$, $\beta=0.5$, both of the infimum of $\Phi^-(\mu)$ and $\Phi^+(\mu)$ can be attained at some finite value. The infimum of $\Phi^-(\mu)$ is positive while the infimum of $\Phi^+(\mu)$ is negative.}
\end{figure}

\paragraph{Acknowledgements.}Zhi-Xian Yu was supported by Shanghai Leading Academic Discipline Project(No. XTKX2012), by Innovation Program of Shanghai Municipal Education Commission (No.14YZ096) and by the Hujiang Foundation of China (B14005).
Zhang's research is supported by the China Scholarship Council under a joint-training program at Memorial University of Newfoundland.


\end{document}